\documentclass[11pt]{amsart}
\usepackage{amsmath, amssymb, amsthm,color}
\usepackage{amssymb,longtable}
\usepackage{amsfonts}
\usepackage{fullpage}
\usepackage{amsmath} 
\usepackage{hyperref} 
\usepackage{latexsym}
\usepackage{array}
\usepackage{amssymb}
\theoremstyle{plain}

\numberwithin{equation}{section}
\newtheorem{thm}{Theorem}[section]

\newtheorem{lemm}[thm]{Lemma}
\newtheorem{cor}[thm]{Corollary}

\theoremstyle{remark}
\newtheorem{rema}[thm]{Remark}

\newcommand{\curl}{{\rm curl}}

\newcommand{\bq}{\begin{equation}}
\newcommand{\eq}{\end{equation}}

\begin{document}

\title[]{Global Regularity for  some Oldroyd  type models}

\author{Tarek Mohamed Elgindi}
\address{The Courant Institute, New-York University, 
251 Mercer Street, New-York,    USA}
\email{elgindi@cims.nyu.edu}
\author{Fr\'ed\'eric Rousset}
\address{IRMAR, Universit\'e de Rennes 1, Campus de Beaulieu, 35042 Rennes cedex, France}



\begin{abstract}

We investigate some  critical models for visco-elastic flows of Oldroyd-B type in dimension two. We use a transformation which exploits the Oldroyd-B structure to prove an $L^\infty$ bound on the vorticity which allows us to prove global regularity for our systems.

\end{abstract}

\maketitle

\section{Introduction}

Over the past several years, there have been many works devoted to visco-elastic equations in dimensions 2 and 3. Much work has been done both in the macro-scopic and micro-scopic regimes. For the most part, with some exceptions, results have been limited to global existence of weak solutions and global well-posedness for smooth solutions with small data. We focus here upon the Oldroyd-B  type models for visco-elastic flow
\begin{align}
\label{equ} &   \partial_{t} u +(u\cdot \nabla) u + \nabla p= \nu \Delta u +  K div(\tau),  \\
\label{eqtau} &  \partial_{t} \tau + (u\cdot \nabla) \tau  + \beta \tau- Q(\nabla u, \tau) = \mu \Delta \tau + \alpha Du,  \\
\label{eqdiv}  &   \nabla \cdot u =0. \end{align}
 We shall assume that the fluid domain is the whole $\mathbb{R}^d$ (we could also consider the $\mathbb{T}^d$ case
  with minor changes).
The vector  $u \in \mathbb{R}^d$  is the velocity of the fluid, $\tau \in \mathcal{S}_{d}(\mathbb{R})$  which is a symmetric tensor is the non-Newtonian part of the stress tensor, $p\in \mathbb{R}$ is the pressure of the fluid, $Du$ is the symmetric part of the velocity gradient
$$ Du = {1 \over 2} ( \nabla u + \nabla u^t)$$
 and
 $Q$ above is a given bilinear form.
  One usually chooses
  $$ Q(\nabla u, \tau)=  \Omega \tau - \tau \Omega + b (Du \tau + \tau Du) $$
  where $\Omega$ is the skew symmetric part of $\nabla u$ and $b$ is a parameter in $[-1, 1]$.
  
  The parameters $K, \, \alpha, \, \mu, \, \nu, \, \beta$ are such that 
  $$ \beta \geq 0, \, \nu \geq 0, \, \mu \geq 0, \, K>0, \, \alpha >0.$$
 There is a growing  literature devoted to these  systems so that it is almost impossible to be exhaustive in this introduction.  We shall thus  only quote the main types of results that have been proven  and  a few  significant works. 
 
 In the case $\mu>0$ and $\nu>0,$ where the system is  fully parabolic,  global well-posedness for strong solutions has been shown by Constantin and Kliegel \cite{4}. In the most interesting case where $\nu>0$ and $\mu=0$ (which is the classical case),  global existence of weak solutions has been shown in the corotational case (with b= 0) only \cite{15}. Many works have  been devoted to  the  proof of global well-posedness  in the case of small data after the work of   F.H. Lin et al. (\cite{12}, \cite{13}, \cite{14}). 
 
 In this paper, we shall be  interested in the case $\nu =0$, $\mu>0$, in dimension two ($d=2$),    thus we shall study a coupling
   between the incompressible Euler equation and a convection diffusion equation for  $\tau$.
    As we shall explain just  below, even when taking $Q=0$ this coupling is critical with respect to the smoothing effect provided by the partial
     parabolic regularization.
     
    Indeed, the case  $\nu>0$,  $\mu>0$ can be considered as a subcritical case in dimension $2$. 
       The smoothing effect provided by the two parabolic equations is enough to ensure the existence of global smooth solutions. 
        This was proven recently by  Constantin and Kliegel \cite{4}.  The proof  is based on  the energy  method and the use of the maximum 
         principle only.  The corotational assumption is not needed.

In $2D$, the case  $\nu>0$,  $\mu=0$ can be considered as critical. Indeed, as already explained, we  
 have global existence of weak solutions (without uniqueness) when $Q$ has a special structure (that is, in the corotational case) and  global existence of smooth solutions with small data.  Neverthless, let us  note that even in the case where $Q\equiv 0,$  the global existence of  smooth  solutions is open and quite challenging. Indeed, consider the following simple model, where we even
  replace the Navier-Stokes equations by the Stokes equations:
$$-\Delta u + \nabla p = div(\tau)$$
$$\tau_{t} +(u\cdot \nabla)\tau = Du$$
$$\nabla \cdot u =0.$$
We see from the first and third equations that $\nabla u$ can be written as a singular integral operator applied to $\tau$, 
and we can rewrite the  system under the form
$$\nabla u= R_{1}\tau$$
$$\tau_{t} +(u\cdot \nabla) \tau = R_{2}\tau$$
for some singular integral operators $R_{1}$, $R_{2}$ (of Riesz transform type).
Because $R_{1}$ and $R_{2}$ are unbounded on $L^\infty,$ it is impossible to derive an $L^\infty$ a-priori estimate on $\tau$ and consequently the well-posedness question seems  difficult unless some special structure is used (see \cite{5}). Such a special structure is not apparent for this system.  The same type  of problem  of borderline regularity occurs if we replace in the above system the stationary Stokes system by
 the evolutionary one.

Still in 2D, the case  $\nu=0$,  $\mu>0$ is also a critical case.  To the best of our knowledge, it has not been studied. For the fully non-linear problem, we prove global well-posedness for small data. Moreover,  we will prove that in the case where $Q=0,$ we have global existence of strong solutions even with large data. In this case, we have a similar situation to the previous case ($\nu>0$, $\mu=0$, $Q=0$). However, a  nice structural property of the system allows us to get around the problem that singular integral operators are not bounded on $L^\infty .$

To conclude this short review, we point out that  in the fully hyperbolic case  where $\nu=0$ and $\mu=0$ in $2D$,  
almost global existence of small smooth solutions has been obtained using dispersive-type estimates in \cite{11}.  Global existence for large solutions is of course  completely open.  

Finally, we can make the connection of  this work with  the works on other systems showing  a critical coupling.
 A lot of works have been devoted recently to the two-dimensional Boussinesq system with partial dissipation \cite{Chae}, \cite{Hou}, \cite{6},  \cite{Danchin}
  for example.
  In particular, a critical coupling has been studied in \cite{7}, \cite{8}, \cite{9}, \cite{10}. The coupling in the Boussinesq system
   is simpler than the one in the Oldroyd models (or in the Magnetohydrodynamics system)  since it is triangular in the sense
    that the vorticity equation is forced by the gradient of the temperature, but that the temperature solves an unforced convection-diffusion
     equation. Consequently, in the case of critical coupling, if one can find a combination of the vorticity and the temperature
      that has better regularity properties, it is rather easy to deduce  an estimate on each individual quantity.
      This is not the case for   the Oldroyd model  (even if $Q=0$)  or the MHD system since they are  fully coupled.

%
%
%

\subsection{Main Results}
We shall first focus on a simplified system for \eqref{equ}, \eqref{eqtau}, \eqref{eqdiv} where we take $Q=0$.
\begin{thm}
\label{th1}
Consider the following system,
\begin{align}
\label{Q=01}&  u_{t}+(u\cdot \nabla) u + \nabla p=  K div(\tau) \\
\label{Q=02}&  \tau_{t} + (u\cdot \nabla) \tau + \beta \tau =  \mu \Delta \tau + \alpha  Du  \\
\label{Q=03}&  \nabla\cdot u =0, \end{align}
with parameters such that $\mu >0$,   and $\alpha \in \mathbb{R}$, $\beta \geq 0$, $K \geq 0$.
 Then for every $(u_{0}, \tau_{0}) \in H^s$ with $s>2$, there exists a unique global solution 
 $ (u, \tau) \in \mathcal{C}([0,  + \infty), H^s)$ with initial data $(u_{0}, \tau_{0})$.
\end{thm}

Note that this is a global existence result of smooth solutions without any size restriction on the data. 
%
%

In the case of the full model \eqref{equ}, \eqref{eqtau}, \eqref{eqdiv} without any particular assumption on the quadratic form $Q$, we shall prove:  
\begin{thm}
\label{th2}
Let us consider  the system
\begin{align}
& u_{t}+(u\cdot \nabla) u + \nabla p= K  div(\tau), \\
&  \tau_{t} + (u\cdot \nabla) \tau  + \beta \tau=  \mu \Delta \tau +  \alpha Du  + Q (\nabla u, \tau), \\ 
& \nabla \cdot u=0,
\end{align}
for any quadratic form $Q$.
 Let us assume that $ K>0, \, \alpha >0$, $\beta>0$, $\mu >0$ then, there exists $\delta>0$ such that for
  $(u_{0}, \tau_{0})\in H^s, \, s>2$ and 
 $$ \|(u_{0}, \tau_{0}) \|_{H^1} + \|(\omega_{0}, \tau_{0} )\|_{B_{\infty, 1}^0} \leq \delta, \quad \omega_{0}= \curl\,  u_{0}$$
  there exists a unique global solution   $(u, \tau) \in \mathcal{C}([0,  + \infty), H^s)$ with initial data $(u_{0}, \tau_{0})$.

\end{thm}

\subsection*{Remark 1.} 

 The difficulty in Theorem 1.1 is getting an $L^\infty$ estimate on the vorticity. Indeed, at first glance it would seem hopeless because the  vorticity equation involves a transport term as well as a non-local term. Therefore application of the maximum principle is not allowed--in fact, it is possible to show that the vorticity \emph{does not} satisfy a maximum principle (see \cite{5}). However, as we will show below, it is possible to find a structural cancellation which allows us to derive $L^\infty$ estimates on the vorticity. The main difference between our system and the original Oldroyd-B system (where the dissipation is in the velocity equation) is the \emph{distribution} of the dissipation--not the amount of dissipation. Indeed, one could distribute the dissipation artificially by putting $(-\Delta)^\eta$ dissipation on $u$ and $(-\Delta)^\theta$ dissipation on $\tau,$ $\eta+\theta\geq 1, \, \theta>0$ and the $L^\infty$ bound on $\omega$ would still be possible. Unfortunately, for the $\theta=0$ case our method does not seem to work.  
\par\
After obtaining the $L^\infty$ estimate on $\omega$, showing the global existence of smooth solutions follows
 using a BKM type argument \cite{2}.
 
 We have chosen to avoid additional technicalities and to state the result by  working in  the most simple functional framework
  that allows to use the BKM argument since we were mostly interested in the global existence of smooth solutions.   
 By using  the arguments of \cite{Vishik}, \cite{Hmidi} in the end of the proof,   it should be possible to handle a   critical regularity for the Euler equation,  
  for example $ u \in  B_{\infty, 1}^1 \cap H^1$, where $B^s_{p,q}$ are the standard Besov spaces (see \cite{1} for example).

\subsection*{Remark 2.}

The difficulty in the proof of  Theorem 1.2 is that there is no clear dissipation of the velocity. Therefore it would seem impossible to prove even a small-data result. Nevertheless, the same structural trick used in the proof of Theorem 1.1 allows us to transfer dissipation from $\tau$ to $u.$ From this we will get that $\|\nabla u\|_{L^2}$ decays exponentially in time which will then allow us to prove global existence for small data.  

Note that the smallness condition that we have stated is independent of the index of regularity that we propagate. 

\bigskip

The paper is organized as follows, in the next section, we shall explain briefly the main steps in the proof of Theorem 1.1 (the proof
 of Theorem 1.2 follows roughly the same scheme).
Next, in section \ref{sectionprelim}, we shall derive some useful commutator estimates. Finally sections
 \ref{sectionth1} and \ref{sectionth2} are devoted to the proof of Theorem 1.1 and 1.2 respectively.

\section{Main Ideas}

In this section we give the main ideas for the proof of Theorem 1.1. To clarify the main ideas, we will take $\beta=0,$ $\mu=\alpha=K=1.$

The first step will be to use the basic energy estimate 
which give
$$\frac{d}{dt}( \int |u|^2 dx + \int |\tau|^2 dx )+\int |\nabla \tau|^2=0.$$
This gives us control on $u$ and $\tau$ in $L^{\infty}_{t}L^{2}_{x}$ and on $\tau$ in $L^{2}_{t}\dot{H}^{1}_{x}.$
Since our system contains the incompressible Euler equation, the estimate on the velocity is not enough
 to get global well-posedness, we basically need to get an estimate of the vorticity in $L^\infty$.
 
  Before, getting the $L^\infty$ estimate on the vorticity, the second step will  be to improve the energy estimate by
   getting an  enstrophy estimate. We shall  show that $$\frac{d}{dt} (\int |\nabla u|^2+\int |\nabla \tau|^2)+\frac{1}{2}\int |\Delta \tau|^2 \leq C \int|\nabla u|^2 \int|\nabla \tau|^2 $$

Using the energy estimate and the enstrophy estimate together along with Gronwall's lemma, we can get bounds on $\Delta \tau$ in $L^2 L^2$ and $\nabla u, \nabla \tau$ in $L^{\infty}_{\text{loc}} L^2 .$  

The main step will be next to get the $L^\infty$ estimate on  the vorticity $\omega= \partial_{1} u_{2} - \partial_{2} u_{1}$.
A consequence of the enstrophy estimates is that $\tau$ belongs to $W^{1-\epsilon,\infty}$ for every $\epsilon>0$ (by the Sobolev imbedding theorem). For the purposes of establishing $L^\infty$ bounds on the vorticity $\omega,$ it will suffice to know that $\tau$ belongs to $W^{\epsilon,\infty}$ for some $\epsilon>0.$
Taking the curl of \eqref{equ} and keeping \eqref{eqtau} as it is, we get the following system:
\begin{align*}
&\omega_{t} + u\cdot \nabla \omega = curl(div(\tau)), \\
&\tau_{t} +(u\cdot\nabla) \tau = \Delta \tau + Du.
\end{align*}
Now, let $R$ denote the singular integral operator which operates on matrices in the following way:
$$R (\cdot) \equiv -(-\Delta)^{-1} curl(div(\cdot)).$$
Note that $R$ is a degree zero singular integral operator. Therefore it maps $W^{\epsilon,\infty}$ continuously to itself for all $\epsilon>0.$
Then taking $R$ of the equation for $\tau$ we get:
$$R(\tau)_{t} +(u\cdot\nabla) R(\tau) + [u\cdot\nabla,R]\tau = curl(div( \tau)) + \omega.$$

Now subtracting the equation for $\omega$ from the equation for $R(\tau)$ we get
$$(\omega-R(\tau))_{t} + u\cdot \nabla (\omega-R(\tau))= -(\omega-R(\tau))-R(\tau) +[u\cdot\nabla,R]\tau.$$
Let $\Gamma=\omega-R(\tau), $
 since we have very good estimates on $\tau$ already, $\Gamma \in L^\infty$ if and only if $\omega \in L^\infty $ with the obvious bounds.    
Note that the structure of the equations is exactly such that the right hand side above is $-\Gamma$ plus some terms which are not too bad.
 The crucial simplification that has occurred is that
  $ R (Du) = \omega$.
  This yields that no singular integral operator applied to $\Gamma$ occurs in the equation for $\Gamma$. Otherwise,  this would have  ruined our $L^\infty$ estimates (again see \cite{5}).   

The only technical point that remains  is to show that the commutator in the right-hand side of the above equation  can be bounded in a good way. This can be done using the fact that $\tau$ has a good estimate. We shall use the following:
$$ |[u\cdot\nabla,R]\tau|_{L^\infty} \lesssim_{\epsilon} (1+|\Gamma|_{L^\infty})|\tau|_{H^{1+\epsilon}}.$$

\section{Preliminary  Tools and Estimates}
\label{sectionprelim}
Throughout the paper, we shall use the notation $\lesssim$ for $\leq C$ where $C$ is an harmless positive constant that my change
 from line to line.  For norms, we use the notation $ \| \cdot \|_{L^p_{T} X}$ with $X$ a normed space for the
  norm $\|\cdot \|_{L^p ([0, T], X)}.$
 
In this section we give some of the commutator estimates which we need. We refer to the book \cite{1} or the introduction
 of \cite{8} for  example for the definition of Besov spaces.
 
Let us set $\mathcal{R}_{i}= \partial_{i}/ |D|$, then we have from Theorem 3.3 in \cite{8} that
\begin{thm}\label{propcom}
Let  $v$ be is  a smooth  divergence-free  vector field. 
\begin{enumerate}
\item For every $(p,r)\in [2,\infty[\times[1,\infty]$ there exists 
a constant $C=C(p,r)$ such that 
$$
\|[\mathcal{R}_{i}, v\cdot\nabla]\theta\|_{B_{p,r}^0}\le C \|\nabla v\|_{L^p}\big(\|\theta\|_{B_{\infty,r}^0}+\|\theta\|_{L^p}\big),
$$
for every smooth scalar function $\theta$.
\item 
For every  $r\in[1,\infty]$  and $\epsilon >0$ there exists 
a constant $C=C(r,\varepsilon)$ such that 
$$
 \|[\mathcal{R}_{i}, v\cdot\nabla]\theta\|_{B_{\infty,r}^0}\le C (\|\omega\|_{L^\infty}+\|\omega\|_{L^2})\big(\|\theta\|_{B_{\infty,r}^\epsilon}+\|\theta\|_{L^2}\big),
 $$
 for every smooth scalar function $\theta$, where $\omega= \mbox{curl } v$.
 \end{enumerate}
 \end{thm}

 We can thus obtain for $R = \Delta^{-1} \mbox{curl }\mbox{div }$ that 
 
 \begin{cor}
 \label{corcom}
 \begin{enumerate}
  For $u$ divergence free, we have
\item For every $(p,r)\in [2,\infty[\times[1,\infty]$ there exists 
a constant $C=C(p,r)$ such that 
$$
\|[ R, u\cdot\nabla]\tau\|_{B_{p,r}^0}\le C \|\nabla u\|_{L^p}\big(\|\tau\|_{B_{\infty,r}^0}+\|\tau\|_{L^p}\big),
$$
\item 
For every  $r\in[1,\infty]$  and $\epsilon >0$ there exists 
a constant $C=C(r,\rho,\varepsilon)$ such that 
$$
 \|[ R, u\cdot\nabla]\tau\|_{B_{\infty,r}^0}\le C (\|\omega\|_{L^\infty}+\|\omega\|_{L^2})\big(\|\tau\|_{B_{\infty,r}^\epsilon}+\|\tau\|_{L^2}\big),
 $$
\end{enumerate}
 
 \end{cor}
 
 \begin{proof}
 We first note that we only have to estimate commutators like
 $$ \mathcal{R}_{i} \mathcal{R}_{j} ( u \cdot \nabla \tau_{k})  - u \cdot \nabla  \mathcal{R}_{i } \mathcal{R}_{j}  \tau_{k}
 = \mathcal{R}_{i} \big( [ \mathcal{R}_{j}, u \cdot \nabla] \tau_{k} \big) +  [\mathcal{R}_{i}, u \cdot \nabla] \mathcal{R}_{j} \tau_{k}.$$
  For (1), the conclusion follows by using  Theorem \ref{propcom}
  and the  estimates
  $$ \| \mathcal{R}_{i} \big( [ \mathcal{R}_{j}, u \cdot \nabla ] \tau_{k} \big)  \|_{B_{p,r}^0} \lesssim \|  [ \mathcal{R}_{j}, u \cdot \nabla ] \tau_{k} \|_{B_{p,r}^0} $$ 
   from  $L^p$, $p <+\infty$ continuity of the Riesz transform 
   and
   $$\| \mathcal{R}_{j}  \tau_{k}\|_{ B_{\infty, r}^0} \lesssim  (\|\tau_{k}\|_{B_{\infty, r}^0} + \|\tau_{k}\|_{L^p} )$$
    which follows from the fact that 
   $$ \| \Delta_{q}   \mathcal{R}_{j}  \tau_{k} \|_{L^\infty} \lesssim \|\Delta_{q} \tau_{k} \|_{L^\infty}, \quad q> -1$$
    and 
    $$ \| \Delta_{-1} \mathcal{R}_{j}  \tau_{k} \|_{L^\infty} \lesssim   \| \Delta_{-1} \mathcal{R}_{j}  \tau_{k} \|_{L^p} \lesssim  \| \mathcal{R}_{j}  \tau_{k} \|_{L^p}.$$
 
  For (2), we use that
  $$ \| \mathcal{R}_{i} \big( [ \mathcal{R}_{j}, u \cdot \nabla ] \tau_{k} \big)  \|_{B_{\infty,r}^0} \lesssim 
 \|  [ \mathcal{R}_{j}, u \cdot \nabla ] \tau_{k}  \|_{B_{\infty,r}^0}  +  \| [ \mathcal{R}_{j}, u \cdot \nabla ] \tau_{k}  \|_{B_{2,r}^0} $$
    and we use again Theorem \ref{propcom}.
 \end{proof}

\section{Proof of Theorem \ref{th1}}
\label{sectionth1}
We shall focus on the proof of a priori estimates, the local well-posedness  in $H^s$ can be obtained by classical arguments. 
 Moreover, the uniqueness is easy since  the velocity is Lipschitz.
 
  In the following we thus consider a sufficiently smooth solution defined on some interval  $[0, T^*]$.
\subsection{Step 1: energy estimates}
Our first estimate is
\begin{lemm}
\label{lem11}
 There exists $C>0$ (which depends on the parameters of the problem) such that
 $$ \|u\|_{L^\infty_{T} L^2} + \| \tau \|_{L^\infty_{T} L^2} + \|\nabla  \tau \|_{L^2_{T} L^2} \leq C_{0} e^{CT}.$$

\end{lemm}
 \begin{proof}
  The proof just follows from standard energy estimates. Integrating by parts, we obtain
$$  \frac{d}{dt}( \int |u|^2 dx + \int |\tau|^2 dx )+ \mu \int |\nabla \tau|^2 \leq ( |K| + | \alpha |) \| \nabla \tau \|_{L^2} \| \|u \|_{L^2} $$
and the result follows. 
 \end{proof}
 
 \subsection{Step 2: energy gradient estimates}
 Let us use the notation
  $\Phi_k$ to denote any function
of the form 
$$
\Phi_k(t)=  C_{0}\underbrace{ \exp(...\exp  }_{k\,times}(C_0t)...),
$$
where $C_{0}$ depends on the involved norms of the initial data and its value may vary from line to line up to some absolute constants. 
We will make an intensive  use (without mentionning it) of  the following trivial facts
$$
\int_0^t\Phi_k(\tau)d\tau\leq \Phi_k(t)\qquad{\rm and}\qquad \exp({\int_0^t\Phi_k(\tau)d\tau})\leq \Phi_{k+1}(t).
$$

  Next, we shall prove that
  \begin{lemm}
\label{lem12}
$$ \| \nabla u\|_{L^\infty_{T} L^2} + \| \nabla  \tau \|_{L^\infty_{T} L^2} + \|\nabla^2  \tau \|_{L^2_{T} L^2} \leq   \Phi_{2}(T).$$

\end{lemm}

\begin{proof}
For the sake of simplicity, we take $K=\mu=\alpha=1,\beta=0$ since their actual values make no difference for the energy estimates. Once more, we are in the case where $Q\equiv 0.$ 
\par\ 
As explained in section 2, we begin by multiplying $\eqref{equ}$ by $\Delta u$ and integrating in space. Then multiply $\eqref{eqtau}$ by $\Delta \tau,$ take the trace, and then integrate in space. We use the following identity: 
$$\int_{\mathbb{R}^{2}} (u\cdot \nabla) u \cdot \Delta u dx =0$$ if $\text{div}(u)=0.$ 

The only term that becomes difficult to estimate is the following cubic term:
$$\int_{\mathbb{R}^{2}} Tr((u\cdot \nabla) \tau \Delta \tau) dx.$$

Using the integration by parts and the divergence free condition we get that we really need to bound the following term:
$$\int |\nabla u| |\nabla \tau|^2 . $$

By the Cauchy-Schwarz inequality, this term can be controlled by 
$$\|\nabla u\|_{L^{2}} \|\nabla \tau\|_{L^{4}}^2 .$$

Using Ladyzhenskaya's inequality we get that 
$$\|\nabla \tau\|_{L^4}^2 \lesssim \|\nabla \tau\|_{L^2}\|\Delta \tau\|_{L^2}.$$ 

Therefore our non-linear term can be controlled in the following way:
$$\int |\nabla u| |\nabla \tau|^2 \lesssim  \|\nabla u\|_{L^2} \|\nabla\tau\|_{L^2}\|\Delta \tau\|_{L^2}$$

Thus, using Young's inequality, we get the following:
\begin{equation}
\label{colibri}\frac{d}{dt} (\int |\nabla u|^2+\int |\nabla \tau|^2)+\frac{1}{2}\int |\Delta \tau|^2 \leq C \int|\nabla u|^2 \int|\nabla \tau|^2
\end{equation}
We  can now  use our bound on $\nabla \tau$ from the energy. Indeed, from Lemma \ref{lem11}, we have that $\nabla \tau \in L^2_{T} L^2$ and it grows at most exponentially in time. Therefore, we
 get from the Gronwall inequality that
  $$\|\nabla u(t)\|_{L^2} + \| \nabla \tau(t) \|_{L^2} \leq \Phi_{2}(t).$$ 
  By using the dissipation term in \eqref{colibri}, we finally get 
$$ \| \nabla u\|_{L^\infty_{T} L^2} + \| \nabla  \tau \|_{L^\infty_{T} L^2} + \|\nabla^2  \tau \|_{L^2_{T} L^2} \leq   \Phi_{2}(T).$$

\end{proof}

\subsection{Step 3: $ L^\infty$ estimates}
We shall next obtain that 
\begin{lemm}
\label{lem13} 
$$
 \| \omega \|_{L^2_{t}L^\infty} +  \| \tau \|_{L^2_{t} B^\varepsilon_{\infty, 1}} \leq \Phi_{2}(t).
$$
for $\varepsilon \in [0, 1)$ with $\omega= \curl u = \partial_{1}  u_{2} - \partial_{2} u_{1}$.
\end{lemm}

\begin{proof}
 By Lemma \ref{lem12} and Besov embeddings, we already have that
 \begin{equation}
 \label{taubesov} \| \tau \|_{L^2_{t} B^\varepsilon_{\infty, 1}} \lesssim \| \tau \|_{L^2_{t} H^2} \leq \Phi_{2}(t)\end{equation}
 and 
 thus it only  remains to estimate $\omega$.
  
  Let us define
  $$ \Gamma=  \mu \omega - K R \tau, \quad R= \Delta^{-1} \mbox{curl div}.$$
  Let us remind that $\mu >0$ so that we will be able to get information on $\omega$ from information on $\Gamma.$
  Let us also notice that
  $$ R Du= \omega, $$
  consequently, we obtain that  $\Gamma$ solves the equation
  \begin{equation}
  \label{gammaequation} \partial_{t} \Gamma + u \cdot \nabla \Gamma=  K  \beta\,  R \tau  - K \alpha\,  \omega + K [ R, u \cdot \nabla ] \tau.
  \end{equation}
   From the maximum principle, we obtain that
   \begin{align*} \|\Gamma (t) \|_{L^\infty}  &  \leq  C_{0} +  C \int_{0}^t \big( \|R \tau \|_{L^\infty} + \| \omega \|_{L^\infty} + \| [ R, u \cdot \nabla ] \tau \|_{L^\infty}
 \big)\, ds \\
 &   \leq  C_{0} +  C \int_{0}^t \big( \|R \tau \|_{L^\infty} + \| \Gamma \|_{L^\infty} + \| [ R, u \cdot \nabla ] \tau \|_{L^\infty}
 .\end{align*}
 To estimate the right hand side, we use that
  $$ \|R \tau \|_{L^\infty} \lesssim  \| R \tau \|_{B_{\infty, 1}^0} \lesssim   \| \tau \|_{B_{\infty, 1}^0}  +  \| \tau \|_{L^2}$$
   and Corollary \ref{corcom} to obtain
   $$  \| [ R, u \cdot \nabla ] \tau \|_{L^\infty} \lesssim  \| [ R, u \cdot \nabla ] \tau \|_{B_{\infty, 1}^0} \lesssim ( \| \omega \|_{L^\infty} + \| \omega \|_{L^2} \big)
    \big(  \| \tau \|_{B_{\infty, 1}^\varepsilon }+  \| \tau \|_{L^2} \big).$$
     By using Lemma \ref{lem11} and Lemma \ref{lem12}, we obtain that
    $$  \|\Gamma (t) \|_{L^\infty}    \leq \Phi_{2}(t) +  C \int_{0}^t  (\Phi_{1}(t) + \| \tau \|_{B_{\infty, 1}^\varepsilon}) \| \Gamma \|_{L^\infty}
     + \int_{0}^t \| \tau \|_{B_{\infty, 1}^\varepsilon}^2.$$
     This yields thanks to \eqref{taubesov}
     $$    \|\Gamma (t) \|_{L^\infty}    \leq   \Phi_{2}(t) +  C \int_{0}^t  (\Phi_{1}(t) + \| \tau \|_{B_{\infty, 1}^\varepsilon}) \| \Gamma \|_{L^\infty}$$
      and hence by using the Gronwall inequality and \eqref{taubesov} again, we finally obtain.
     \begin{equation}
     \label{GammaLinfty} \| \Gamma (t)  \|_{L^\infty} \leq  \Phi_{3}(t).\end{equation}
       From the definition of $\Gamma$, this yields
       $$ \| \omega \|_{L^2_{t} L^\infty} \lesssim \|\Gamma \|_{L^2_{t} L^\infty} + \| \tau \|_{L^1_{t}L^\infty} \leq \Phi_{3}(t)$$
       by using \eqref{taubesov} and \eqref{GammaLinfty}.
      
\end{proof}

\subsection{Step 4: $H^s$ estimates}.
\begin{lemm}
\label{lem14}  For $s>2$, we have
$$ \|u \|_{L^\infty_{t} H^s} + \| \tau \|_{L^\infty H^s} \leq \Phi_{5}(t).$$
\end{lemm}
\begin{proof}
We perform a classical $H^s$ energy estimate for the system \eqref{Q=01}, \eqref{Q=02}, \eqref{Q=03}. Applying the operator $\Lambda^s= ( 1 + |D|^{2s})^{1 \over 2}$ to the system, 
 we obtain in classical way
 \begin{equation}
 \label{Hsestimate1} {d \over dt} \big( \| u \|_{H^s}^2 +  \| \tau \|_{H^s}^2  \big)+  \mu \| \nabla \tau \|_{H^s}^2
  \lesssim  \|u \|_{H^s}^2 + \| \nabla u \|_{L^\infty} \|u\|_{H^s}^2 + \sum_{ij} \| \Lambda^s(u \tau_{ij}) \|^2.
  \end{equation} 
  where for the last term, we have used that
  $$ \Big| \int \Lambda^s (u \cdot \nabla \tau) : \Lambda^s \tau \Big|=\Big| \sum_{ij} \Lambda^s \nabla \cdot (\tau_{ij} u)  \Lambda^s \tau_{ij} dx \Big|
  = \Big| \sum_{ij} \Lambda^s(\tau_{ij} u)  \cdot \nabla \Lambda^s \tau_{ij} \Big| \leq C_{\eta}\| \Lambda^s (u \tau_{ij)} \|_{L^2}^2  + \eta \| \nabla \Lambda^s \tau \|_{L^2}^2$$
  and hence that by choosing $\eta>0$ sufficiently small the last term can be absorbed in the left hand side of  \eqref{Hsestimate1}.
 
 Next, we use that
 $$ \| u \tau_{ij}\|_{H^s} \leq \|u\|_{L^\infty} \|\tau\|_{H^s} + \| \tau \|_{L^\infty} \|u \|_{H^s}  \leq (\|u\|_{L^2} + \| \omega \|_{L^\infty}) \|\tau\|_{H^s} + \| \tau \|_{L^\infty}
   \|u\|_{H^s}\big)$$
    and that
  $$ \| \nabla u \|_{L^\infty} \leq \| \omega \|_{L^\infty} \log \big( e + \|u\|_{H^s} \big).$$
   to obtain that
   $$ {d \over dt} \big( \| u \|_{H^s}^2 + 
    \| \tau \|_{H^s}^2 \big)  \leq \ \| \omega \|_{L^\infty} \log \big( e + \|u\|_{H^s} \big) \|u\|_{H^s}^2 + 
    \big( \|u\|_{L^2}^2 + \| \omega \|_{L^\infty}^2 +  \|\tau \|_{L^\infty}^2\big) \big( \|u \|_{H^s}^2 + \| \tau \|_{H^s}^2 \big).$$
 By using Lemma \ref{lem13} and Lemma \ref{lem11}, we get the result from the classical Beale-Kato-Majda argument.
\end{proof}

\section{Proof of Theorem 1.2}
\label{sectionth2}

Again,  we  shall prove a priori estimates for a sufficiently smooth solution defined on $[0, T]$, $C$ will always denote a number that
  only depends on the physical constants of the system.
 \subsection{Step 1: energy estimates}

\begin{lemm}
\label{lem21}
  There exists $C>0$,  such that for every $T>0$ and every sufficiently smooth solution on $[0,T]$ we have the estimate
  \begin{multline*}
  {d \over dt} { 1 \over 2} \Big(  \alpha \|u\|_{L^2}^2 + K \| \tau \|_{L^2}^2  +  \alpha \| \nabla u\|_{L^2}^2 + K \| \nabla  \tau \|_{L^2}^2  \Big) +  {\beta K  } \| \tau \|_{L^2}^2    + \big( {\mu K \over 2 }
   + \beta K\big)\| \nabla \tau \|_{L^2}^2 +    {\mu K \over 2}\| \nabla^2 \tau \|_{L^2}^2 \\
   \leq  C \big( \| \tau \|_{L^\infty}^2 +   { C_{Q}^2 K \over \mu}\|  \tau \|_{H^1}^2 \big)\| \nabla u \|_{L^2}^2.
 \end{multline*}
\end{lemm}
\begin{proof}
 We shall first  compute an energy estimate 
  quite similar to the one of Lemma \ref{lem11}, but we compute
$$ { d \over dt } { 1 \over 2} \big(  \alpha \|u\|_{L^2}^2 + K \| \tau \|_{L^2}^2 \big)
$$
  and we note that
   $$ \int \alpha K \mbox{div } \tau \cdot u + \alpha K \int Du : \tau= 0.$$
 We thus find
 $${ d \over dt } { 1 \over 2} \big(  \alpha \|u\|_{L^2}^2 + K \| \tau \|_{L^2}^2 \big) +  \mu K \| \nabla \tau \|_{L^2}^2 +  \beta K \| \tau \|_{L^2}^2
   = K \int Q (\nabla u, \tau) : \tau \leq  C_{Q}K  \| \nabla u \|_{L^2} \| \tau \|_{L^4}^2.$$
    and hence, by using the Sobolev inequality
    \begin{equation}
    \label{sobolevL4}
    \| u \|_{L^4}^2 \lesssim \|u\|_{L^2} \| \nabla u \|_{L^2}, 
      \end{equation}
     and the Young inequality, we get 
  \begin{equation}
  \label{th2energie}
   {d \over dt } { 1 \over 2} \big(  \alpha \|u\|_{L^2}^2 + K \| \tau \|_{L^2}^2 \big) +  {\mu K \over 2 } \| \nabla \tau \|_{L^2}^2 +  {\beta K  } \| \tau \|_{L^2}^2
   \leq  C_{Q}^2{  K \over 2  \mu }  \| \nabla u \|_{L^2}^2 \|\tau \|_{L^2}^2.
  \end{equation}
   
   We shall next perform an estimate similar to the one of Lemma \ref{lem12}:
 $$ { d \over dt } { 1 \over 2} \big(  \alpha \| \nabla u\|_{L^2}^2 + K \| \nabla  \tau \|_{L^2}^2 \big) 
   +  \mu K \| \nabla^2  \tau \|_{L^2}^2 +  \beta K \| \nabla  \tau \|_{L^2}^2  \leq K \int Q (\nabla u, \tau) : \Delta  \tau - K \int u \cdot \nabla \tau : \Delta \tau$$
    and we use that
    $$   \int Q (\nabla u, \tau) : \Delta \tau \leq  \| \Delta \tau \|_{L^2} \| \tau \|_{L^\infty} \| \nabla u \|_{L^2}$$
     and that after an integration by parts
     $$  - K \int u \cdot \nabla \tau : \Delta \tau \leq  CK  \| \nabla u \|_{L^2} \| \nabla \tau \|_{L^4}^2  \leq C K \| \nabla u \|_{L^2} \| \nabla   \tau \|_{L^2}
      \| \nabla^2 \tau \|_{L^2}.$$
       Consequently, we obtain from the Young inequality that
 \begin{equation}
  \label{th2enstrophy}
   {d \over dt } { 1 \over 2} \big(  \alpha \|\nabla u\|_{L^2}^2 + K \| \nabla \tau \|_{L^2}^2 \big) +  {\mu K \over 2 } \| \nabla^2 \tau \|_{L^2}^2 +  {\beta K  } \|\nabla  \tau \|_{L^2}^2
    \leq   C \big( \| \tau \|_{L^\infty}^2 +   { C_{Q}^2 K \over \mu}\| \nabla \tau \|_{L^2}^2 \big)\| \nabla u \|_{L^2}^2.
  \end{equation}
  \end{proof}

  \subsection{Step 2: detecting the enstrophy dissipation}
  
  \begin{lemm}
\label{lem22}
  There exists $C>0$, and $c_{0}>0$  such that for every $T>0$ and every sufficiently smooth solution on $[0,T]$ such that
  $ \| \nabla  u \|_{L^2} \leq c_{0}$, 
   we have the estimates
 $$ \| (u(t), \tau(t)) \|_{L^\infty_{T}H^1} +  
  \| \nabla u\|_{L^2_{T}H^1} + \| \tau \|_{L^2_{T} H^2 } \leq  C (\|u_{0}\|_{H^1} + \| \tau_{0} \|_{H^1}).$$
\end{lemm}

\begin{rema}
\label{remarkweak}
Note that assuming that the initial data are small in $H^1$, Lemma \ref{lem22}, yields a global closed  $H^1$ estimate. We could use this estimate
 to construct a global weak solution (for small data) for the system for which the vorticity is merely in $L^2$ (that is to say like a  DiPerna-Majda type
 solution of the Euler equation \cite{DiPerna-Majda}). 
\end{rema}
\begin{proof}
  We shall now  use the structure of the system in order to detect a damping phenomenon on $\omega= \mbox{curl }u$.
   Let us set
   $$ \Gamma =  \mu  \omega  - K R  \tau.$$
   We can rewrite the  equation \eqref{gammaequation} for $\Gamma$ under the form
   \begin{equation}
   \label{Gammadamping} \partial_{t} \Gamma + u \cdot \nabla \Gamma  + {K \alpha \over \mu } \Gamma = (K \beta - {K \alpha \over \mu}) R \tau +  K [ R, u \cdot \nabla ] \tau - K R \big( Q(\nabla u , \tau) \big).\end{equation}
    Note that since $ K \alpha / \mu >0$ we have detected some damping. A standard $L^2$ energy estimate for this equation yields
    $$ {d \over dt} {1 \over 2 }  \| \Gamma \|_{L^2}^2  +   {1 \over 2 }  {K \alpha \over \mu }  \| \Gamma \|_{L^2}^2 \leq   { \mu \over \alpha K} \big( K \beta - 
   {K \alpha \over \mu} \big)^2 \| R \tau \|_{L^2}^2 + \big( K \|\nabla u \|_{L^2} \| \tau \|_{L^\infty} +   K  \| [ R, u \cdot \nabla ] \tau \|_{L^2}\big) \| \Gamma \|_{L^2}$$
   where we have used the $L^2$ continuity of the Riesz transform to estimate the term involving $Q$.
    To handle the last term in the right hand side, we use  Corollary \ref{corcom} (1) (with $p=2, \, r= 2$) to get
    $$  \| [ R, u \cdot \nabla ] \tau \|_{L^2}  \leq  C \| \nabla u \|_{L^2} \big( \| \tau \|_{B_{\infty, 2}^0} +  \| \tau \|_{L^2} \big) \leq C  \| \nabla u \|_{L^2} \| \tau \|_{H^1} $$
    where the last estimate is a consequence of  the embedding $H^1 \subset B^0_{\infty, 2}$.
      Consequently, by using the Young inequality, we  obtain:
   \begin{equation}
   \label{th2damping}
     {d \over dt} {1 \over 2 }  \| \Gamma \|_{L^2}^2  +   {1 \over 2 }  {K \alpha \over \mu }  \| \Gamma \|_{L^2}^2 \leq   { \mu \over \alpha K} \big( K \beta - 
   {K \alpha \over \mu} \big)^2 \| R \tau \|_{L^2}^2 +   \big( \| \tau \|_{H^1}^2 + \|\tau\|_{L^\infty}^2 \big)   \| \nabla u  \|_{L^2}^2   \| \tau \|_{H^1}^2. 
   \end{equation}
    To conclude, we can combine the estimate of Lemma \ref{lem21} and \eqref{th2damping}.
     Let us define
     $$ N(u,\tau)=  M \big(   \alpha \|u\|_{L^2}^2 + K \| \tau \|_{L^2}^2 \big) + M \big(  \alpha \|\nabla u\|_{L^2}^2 + K \| \nabla \tau \|_{L^2}^2 \big) 
      +    \| \Gamma \|_{L^2}^2,$$
      with $M>0$ sufficiently large, then there exists $c_{0}>0$ that depends only on the physical constants of the problem  such that
      $$ { d \over dt } { 1 \over 2} N(u, \tau) + c_{0} \big(  \| \tau \|_{H^2}^2 + \| \nabla u \|_{L^2}^2 \big)
       \leq   C\big( \|\tau \|_{L^\infty}^2 +   \| \tau \|_{H^1}^2  \big) \| \nabla u \|_{L^2}^2.$$
        Consequently, by using the Sobolev embedding $H^2 \subset L^\infty$ and by  assuming that
        $$  C    \| \nabla u \|_{H^1}^2  \leq  {c_{0} \over 2}, $$
        we find that
       $${ d \over dt } { 1 \over 2} N(u, \tau) + {c_{0} \over 2 }  \big(  \| \tau \|_{H^2}^2 + \| \nabla u \|_{L^2}^2 \big)
       \leq   0$$
       hence in particular, 
       $$N(u(t), \tau(t)) \leq N(u_{0}, \tau_{0}),  \,   \| \tau \|_{L^2_{t}H^2}^2 + \| \nabla u \|_{L^2_{t}L^2}^2 \leq  C N(u_{0}, \tau_{0}).$$
     \end{proof}   
     
     \subsection{Estimate of the vorticity in a  strong norm}
     \begin{lemm}
\label{lem23}
  There exists $C>0$, and $c_{0}>0$  such that for every $T>0$ and every sufficiently smooth solution on $[0,T]$ such that
  $ \| \nabla  u \|_{L^2}  \leq c_{0}$ and
   \begin{equation}
   \label{hyp2}  \forall t, \, s, \,  0 \leq s \leq t \leq T, \int_{s}^t \| \nabla v \|_{L^\infty} \leq  {1 \over 2 } + c_{0}(t-s), \quad \| \Gamma(t)\|_{B_{\infty, 1}^0}
    \leq c_{0}
   \end{equation}
   where $\Gamma =  \mu \omega -K R \tau$, then 
   we have the estimate
 $$ \| \Gamma   \|_{L^\infty_{T} B_{\infty, 1}^0} \leq  C (\|(u_{0}, \tau_{0})\|_{H^1} + \|( \omega_{0}, \tau_{0})\|_{B_{\infty, 1}^0})
  .$$
\end{lemm}
Note that here, we shall estimate the stronger $B_{\infty, 1}^0$ norm rather than the $L^\infty$ norm because of the Riesz transform
 of the  $Q$ term.
 \begin{proof}
 We first observe that for $\varepsilon \in (0, 1)$, 
 $$  \| \tau  \|_{ L^2_{T}B_{\infty, 1}^\varepsilon } \lesssim  \delta_{0}.$$
  In the proof, we shall denote by $\delta_{0}$ the size of the initial data in the naturally involved norms.
 
  We use again the  equation \eqref{Gammadamping}.
  Let us set ${c_{0} \over 2 }= {K \alpha \over \mu}$. We shall use the following estimate:
   (see [1] for example)
   \begin{lemm}
    For $v$ a divergence free vector field, the solution of the convection equation
    $$ \partial_{t} f + v \cdot \nabla f = F, \quad f_{/t=0}= f_{0}$$
     satisfies the estimate
     $$ \|f(t)\|_{B_{\infty, 1}^0} \leq e^{\int_{0}^t \|\nabla u \|_{L^\infty}}  \|f_{0}\|_{B_{\infty, 1}^0}+ \int_{0}^t e^{\int_{s}^t \| \nabla u \|_{L^\infty}} \, \| F(s) \|_{B_{\infty, 1}^0}\, ds.$$
   \end{lemm}
   From this lemma, get by using the damping term in \eqref{Gammadamping}:
   \begin{equation}
   \label{Gammaestimate2}
    \| \Gamma (t) \|_{B_{\infty, 1}^0} \leq  e^{- 2c_{0} t + \int_{0}^t \|\nabla u \|_{L^\infty}}  \| \Gamma_{0}\|_{B_{\infty, 1}^0}
    + \int_{0}^t e^{- 2c_{0}(t-s) + \int_{s}^t \| \nabla u \|_{L^\infty}} \, \| F(s) \|_{B_{\infty, 1}^0}\, ds
    \end{equation}
    with
    $$ F= (K \beta - {K \alpha \over \mu}) R \tau +  K [ R, u \cdot \nabla ] \tau - K R \big( Q(\nabla u , \tau) \big):= F_{1}+ F_{2}+ F_{3}.$$
     From \eqref{hyp2}, we get that for $t \in [0, T]$:
     $$  \| \Gamma (t) \|_{B_{\infty, 1}^0} \leq  e^{- c_{0} t }   \delta_{0} + 
     \int_{0}^t e^{- c_{0}(t-s) } \, \| F(s) \|_{B_{\infty, 1}^0}\, ds.$$
     Next, we observe that 
     \begin{equation}
     \label{F1estimate} \int_{0}^t  e^{- c_{0}(t-s) } \, \| F_{1}(s) \|_{B_{\infty, 1}^0}\, ds \leq e^{-c_{0} t} \int_{0}^t e^{c_{0} s}  \big( \|\tau(s) \|_{L^2}
      + \| \tau (s) \|_{B_{\infty, 1}^\varepsilon} \, ds 
      \leq e^{-c_{0} t} \int_{0}^t  e^{c_{0} s} \| \tau (s) \|_{ H^2} \big) \, ds \leq C \delta_{0}
      \end{equation}
       by using Cauchy-Schwarz and Lemma \ref{lem22}.

 In a similar way, from (2) in Corollary \ref{corcom} and Lemma \ref{lem22}, we obtain
 $$ \int_{0}^t e^{- c_{0}(t-s) } \, \| F_{2}(s) \|_{B_{\infty, 1}^0}\, ds \leq \int_{0}^t  e^{- c_{0}(t-s) }
  \big(  \|\Gamma (s) \|_{L^\infty} +  \delta_{0} g(t)  \big) \big( \| \tau \|_{B_{\infty, 1}^\varepsilon} + \delta_{0} g(t) \big)\, ds$$
  with $g(t)$ an $L^2$ function.
   This yields
   \begin{equation}
   \label{F2estimate} \int_{0}^t e^{- c_{0}(t-s) } \, \| F_{2}(s) \|_{B_{\infty, 1}^0}\, ds \leq C \delta_{0} +  C \delta_{0} c_{0} \sup_{[0, t]} \| \Gamma \|_{B_{\infty, 1}^0}
   \end{equation}
   by using  again Cauchy-Schwarz and the assumption \eqref{hyp2}.
  To estimate the term involving $F_{3}$, we use that
 $$ \|F_{3} \|_{B_{\infty, 1}^0} \leq   \| Q( \nabla u, \tau) \|_{L^2} + \|  Q(\nabla u, \tau)\|_{B_{\infty, 1}^0} 
   \leq   \big( \| \nabla u \|_{L^2} + \| \nabla u \|_{B_{\infty, 1}^0} \big) \| \tau \|_{B_{\infty, 1}^\varepsilon}
   $$
   and hence by using again Lemma \ref{lem22}, we find
   $$  \|F_{3} \|_{B_{\infty, 1}^0} \leq 
   \big(  \delta_{0} g(t)   \| \Gamma \|_{B_{\infty, 1}^0} + \|\tau \|_{B_{\infty,1}^\varepsilon} \big)  \| \tau \|_{B_{\infty, 1}^\varepsilon}.
$$
Consequently,  we also obtain
\begin{equation}
\label{F3estimate}
 \int_{0}^t e^{- c_{0}(t-s) } \, \| F_{3}(s) \|_{B_{\infty, 1}^0}\, ds \leq  C  \delta_{0}  +  C \delta_{0} c_{0} \sup_{[0, t]} \| \Gamma \|_{B_{\infty, 1}^0}. 
\end{equation}

By combining \eqref{Gammaestimate2} and \eqref{F1estimate}, \eqref{F2estimate}, \eqref{F3estimate}, this yields
$$ \| \Gamma (t) \|_{B_{\infty, 1}^0} \leq    C \delta_{0} + 
     C \delta_{0} c_{0} \sup_{[0, t]} \| \Gamma \|_{B_{\infty, 1}^0}
     $$
  and the result follows by choosing  for $c_{0}$ sufficiently small.  
 \end{proof} 
   
   \begin{cor}
   \label{corfinal}
   Under the assumptions of Lemma \ref{lem23},  
we have for  every $t, \, s$, $0 \leq s \leq t \leq T$, 
   $$ \int_{s}^t \| \nabla u  \|_{L^\infty} \leq   C ( \delta (t-s) + \delta ).$$ 
   \end{cor}
   \begin{proof}
    We note that
    $$ \int_{s}^t \| \nabla u  \|_{L^\infty} \leq \int_{s}^t \big( \|u \|_{L^2} +  \| \nabla u \|_{B_{\infty, 1}^0})  \leq \int_{s}^t \big( \|(u, \tau) \|_{L^2} +  \| \Gamma  \|_{B_{\infty, 1}^0}) + \|\tau \|_{B_{\infty, 1}^0}.$$
    Consequently, by using Lemma \ref{lem22} and Lemma \ref{lem23}, we obtain
    $$  \int_{s}^t \| \nabla u  \|_{L^\infty} \leq C( \delta_{0} (t-s) + \delta_{0} \sqrt{t-s}) \leq C( \delta_{0} +    \delta_{0} (t-s))$$
    where the last estimate follows from the Young inequality.
   \end{proof}
   
   \subsection{Proof of Theorem 1.2}
   The proof follows by bootstrap. Indeed, if the initial data $(u_{0},\tau_{0})$ is of size less than or equal to $\delta$ in  $H^1 \times H^1$ and $B^{1}_{\infty,1} \times B^{0}_{\infty,1}$  with $\delta$ is small enough, then the solution will remain small for a short time afterwards in the same spaces. This follows from the local well-posedness. Using Lemmas 5.2 and 5.4 and Corollary 5.6 we see that the solution will thus 
    be globally defined  and  remains small in these spaces.  We  can  then  propagate the  $H^s$ regularity globally in time. This follows from observing that controlling $|\nabla u|_{ L^{1}L^\infty}$ is the key to prevent blow-up in $H^s \times H^s ,$ \`a  la  Beale-Kato-Majda.\qed

  \bigskip 
   
  \subsection*{Acknowledgements}
  T.M. Elgindi is partially supported by NSF grant \# XXXX grant. This work was carried out during a visit of the second author to the Courant Institute, he  thanks P. Germain and N. Masmoudi for the invitation.

\end{document}